\documentclass[10pt,reqno]{amsart}
\usepackage[margin=1in]{geometry}
\usepackage{amssymb}
\usepackage{amsmath}
\usepackage{amsthm}
\usepackage{float}
\usepackage{microtype}
\usepackage{mathtools}
\usepackage[hidelinks]{hyperref}
\usepackage[labelsep=none]{caption}

\def\quotes#1{``#1''}
\def\o#1{\operatorname{#1}}
\def\p#1{\left( #1 \right)}
\def\set#1{\left\{ #1 \right\}}

\def\Z{\mathbb{Z}}
\def\Q{\mathbb{Q}}

\theoremstyle{plain}
\newtheorem{thrm}{Theorem}[section]
\newtheorem{lemm}[thrm]{Lemma}

\newtheorem{prop}[thrm]{Proposition}

\theoremstyle{definition}

\newtheorem{defn}[thrm]{Definition}

\allowdisplaybreaks

\title{Curious Squares}
\author{Neelima Borade and Jacob Mayle}

\begin{document}

\begin{abstract} 
A \textit{curious number} is a palindromic number whose base ten representation has the form \linebreak $a \ldots a b \ldots b a \ldots a$. In this paper, we  determine all curious numbers that are perfect squares. Our proof involves reducing the search for such numbers to several single variable families. From here, we complete the proof in two different ways. The first approach is elementary, though somewhat ad hoc. The second entails studying integral points on elliptic curves and is more systematic.
\end{abstract}

\maketitle

\vspace{-.5cm}

\section{introduction}

Ian Stewart begins his popular recreational mathematics book \textit{Professor Stewart's Hoard of Mathematical Treasures} \cite{St2010} with the following \quotes{calculator curiosity}:
\begin{align*}
(8 \times 8) + 13 &= 77 \\
(8 \times 88) + 13 &= 717 \\
(8 \times 888) + 13 &= 7117 \\
(8 \times 8888) + 13 &= 71117 \\
(8 \times 88888) + 13 &= 711117 \\
(8 \times 888888) + 13 &= 7111117 \\
(8 \times 8888888) + 13 &= 71111117 \\
(8 \times 88888888) + 13 &= 711111117.
\end{align*}
The numbers on the right hand sides of the equalities above are examples of what we call \textit{curious} numbers.

\begin{defn} For nonnegative integers $m,n$, an integer is \textit{$(m,n)$-curious} if its base ten representation is
\[ \underbrace{a \ldots a}_{m} \underbrace{b \ldots b}_{n} \underbrace{a \ldots a}_{m} \]
for some integers $1 \leq a \leq 9$ and $0 \leq b \leq 9$. A number is called \textit{curious} if it is $(m,n)$-curious for some $m,n$.
\end{defn}

We see that none of the numbers in Stewart's \quotes{calculator curiosity} are perfect squares. This follows from the observation that each number in the family $71\ldots17$ is congruent to $2$ modulo $5$, a quadratic non-residue.

The question of which repdigits (i.e. $(0,n)$-curious numbers) are perfect squares is a pleasant exercise in elementary number theory. There are 28 repdigits of length at most three. Among those, only $0,1,4,$ and $9$ are perfect squares. No repdigit of length at least four is a perfect square since one may verify that each of
\[ 1111, \, 2222, \, 3333, \, 4444, \, 5555, \, 6666, \, 7777, \, 8888, \, \text{and } 9999 \]
is a quadratic non-residue modulo 10000.

The general question of which repdigits are perfect powers is known as \textit{Obl\'ath's problem}. After R. Obl\'ath's  \cite{Ob1956} partial solution, the problem was fully solved by Y. Bugeaud and M. Mignotte \cite{BM19999} using bounds on $p$-adic logarithms. Variations on Obl\'ath's problem have been studied by several authors. For instance, A. Gica and L. Panaitopol \cite{GP2003} determined all perfect squares among the \textit{near repdigits}, which are integers for which all but a single digit are equal. Recently, B. Goddard and J. Rouse \cite{GR2017}  determined the perfect squares that may be written as the sum of two repdigits. Whereas other authors relied on techniques related to Pell equations to settle difficult cases, Goddard and Rouse used a powerful technique involving elliptic curves.

In this paper, we determine which curious numbers are perfect squares. Specifically, we prove the following.
\begin{thrm} \label{main-thrm} The curious numbers that are perfect squares are  $0, \, 1,\, 4,\, 9,\, 121,\, 484,\, 676,$ and $44944.$
\end{thrm}
Broadly speaking, we prove Theorem \ref{main-thrm} as follows. First, in \S \ref{alg-reform}, we recast the problem in algebraic terms. Next, in \S \ref{mod107}, we reduce modulo $10^7$ to narrow down our search for curious squares to thirteen single variable families. From here, we complete the proof in two distinct ways. The first approach, given in \S \ref{modfamilies}, proceeds by studying the families via modular arithmetic. This approach is somewhat ad hoc, whereas our second approach, given in \S \ref{elliptic},  is more  systematic. Here we study the families via integral points on elliptic curves, along the lines of Goddard and Rouse \cite[\S4]{GR2017}.

\section{Algebraic reformulation} \label{alg-reform}

In this section, we recast our problem in algebraic terms. To do so, we start by developing an algebraic expression for curious numbers. Note the following standard expression for repdigits,
\begin{equation} \label{repdigit} a_m \coloneqq \underbrace{a \ldots a}_m = a \cdot \frac{10^m-1}{10-1}. \end{equation}
Now observe that
\begin{equation} \label{cur-notation}  a_mb_na_m \coloneqq
\underbrace{a \ldots a}_{m} \underbrace{b \ldots b}_{n} \underbrace{a \ldots a}_{m} = 10^{m+n} \cdot a_m + 10^m \cdot b_n + a_m.
\end{equation}
In the shorthand $ a_mb_na_m$, juxtaposition denotes concatenation rather than multiplication. Though this notation is more compact, it may cause confusion. In such instances, we opt to use the longhand notation instead. Now, upon combining equations (\ref{repdigit}) and (\ref{cur-notation}), we obtain an expression for curious numbers, 
\begin{equation} \label{exp1}
a_mb_na_m
= a \cdot 10^{m+n} \cdot \frac{10^m-1}{10-1} + b \cdot 10^m \cdot \frac{10^n-1}{10-1} + a\cdot \frac{10^m-1}{10-1}. \\
\end{equation}

To streamline matters, we define the integers
\begin{equation} \label{MN}
    M_{a,b,m} \coloneqq 10^m \cdot (a-b) - a \quad \text{ and } \quad
    N_{a,b,m} \coloneqq  10^m (a \cdot 10^m + b - a).
\end{equation}
After some regrouping, (\ref{exp1}) becomes
\begin{equation} \label{cur-alg-clean}
a_m b_n a_m
= \frac{1}{9} \p{N_{a,b,m} \cdot 10^n + M_{a,b,m}}\!.
\end{equation}

We're interested in determining the collection of $a_mb_na_m$ that are perfect squares. Thus, by multiplying (\ref{cur-alg-clean}) through by 9, we record that our problem reduces to solving the equation 
\begin{equation} \label{alg-eq} (3y)^2 = N_{a,b,m} \cdot 10^n + M_{a,b,m}  \end{equation}
in the nonnegative integers $a,b,m,n,y$ with the restrictions that $1 \leq a \leq 9$ and $0 \leq b \leq 9$.

\section{Narrowing search to several single variable families} \label{mod107}

Let $\mathcal{C}_{m,n}$ denote the set of $(m,n)$-curious numbers and write $\mathcal{C} = \bigcup_{m,n \geq 0} \mathcal{C}_{m,n}$ for the set of all curious numbers. Let $\mathcal{S}$ denote the set of all perfect squares. For a positive integer $k$, define the set
\[ \mathcal{S}_k \coloneqq \set{s^2 \in \mathcal{S} : s \leq k}\!. \]
Finally, we denote the reduction modulo $10^7$ map by
\[\pi: \Z \longrightarrow \Z/10^7\Z.\]

In the notation above, our ultimate objective is to determine the intersection $\mathcal{C} \cap \mathcal{S}$. We start by computing the intersection $\pi(\mathcal{C}) \cap \pi(\mathcal{S})$. Here, our choice to reduce modulo $10^7$ is somewhat arbitrary. In fact, what follows works provided that we reduce modulo $10^k$ for any $k \geq 4$. However, we choose $k = 7$ because this is the smallest exponent for which the number of families that we will need to consider in the next section is minimal.

Observe that $\pi(\mathcal{C})$ and $\pi(\mathcal{S})$ may be realized as
\[ \pi(\mathcal{C}) = \bigcup_{\substack{(m,n) \\ 0 \leq m \leq 7 \\ 0 \leq n \leq 7-m}} \pi(\mathcal{C}_{m,n}) \quad \text{and} \quad \pi(\mathcal{S}) = \pi(\mathcal{S}_{10^{7}}). \]
As the sets $\mathcal{C}_{m,n}$ and $\mathcal{S}_{10^{7}}$ are finite and of reasonable size, it is straightforward to determine $\pi(\mathcal{C})$ and $\pi(\mathcal{S})$ via machine computation. We do so using SageMath \cite{SAGE}, and find that they have the intersection
\begin{align*} 
\pi(\mathcal{C}) \cap \pi(\mathcal{S}) = 
\{
&0,
 1,
 4,
 9,
 121,
 161,
 484,
 656,
 676,
 929,
 969,
 1001,
 1441,
 1881,
 4004,
 4224,
 5225,
 6116, \\
& 6336,
 9009,
 9449,
 9889,
 10001,
 14441,
 18881,
 40004,
 44544,
 44644,
 44944,
 52225,
 67776, \\
& 90009,
 94449,
 98889,
 100001,
 144441,
 188881,
 400004,
 442244,
 447744,
 522225,
 655556, \\
& 677776, 
 900009,
 944449,
 988889,
 1000001,
 1444441,
 1888881,
 2222224,
 2222225,
 2222244, \\
& 3333444, 
 4000004,
 4222224,
 4222244,
 4333444,
 4422244,
 4433344,
 4433444,
 4441444, \\
& 4444441,
 4444449,
 4445444,
 4449444,
 4477444,
 4777444,
 4777744,
 5222225,
 5555556, \\
& 6555556,
 7777444,
 8888881, 
 8888889,
 9000009,
 9444449,
 9888889
\}.
\end{align*}

We have that $\pi(\mathcal{C} \cap \mathcal{S}) \subseteq \pi(\mathcal{C}) \cap \pi(\mathcal{S})$. Those elements of $\pi(\mathcal{C}) \cap \pi(\mathcal{S})$ whose preimage under the map $\mathcal{C} \overset{\pi}{\longrightarrow} \Z/10^7\Z$ consists of a single non-square integer are not members of $\pi(\mathcal{C} \cap \mathcal{S})$. For instance, the preimage of $4333444$ under $\mathcal{C} \overset{\pi}{\longrightarrow} \Z/10^7\Z$ is the singleton $\set{444333444}$. Since $444333444$ is not a perfect square, we ascertain that $44333444 \not\in \pi(\mathcal{C} \cap \mathcal{S})$. Proceeding in this way, we find that
\begin{align*}
\pi(\mathcal{C} \cap \mathcal{S}) \subseteq
\{
& 0,
 1,
 4,
 9,
 121,
 484,
 676,
 44944,
 2222224,
 2222225,
 2222244,
 3333444, \\
& 4444441,
 4444449,
 5555556,
 7777444,
 8888881,
 8888889
\}.
\end{align*}
Taking the preimage of the above set inclusion under the map $\mathcal{C} \overset{\pi}{\longrightarrow} \Z/10^7\Z$, we deduce that 
\begin{align} \label{CSpos}
\mathcal{C} \cap \mathcal{S} \subseteq 
\{
&0,
1,
4,
9,
121,
 484,
 676,
 44944,
 10 \ldots 01,
 14\ldots41, 
 18\ldots81, 
 40 \ldots 04,
 42\ldots24, 
 442\ldots244, \\
& 4443\ldots3444,
 4447 \ldots 7444,
 52\ldots25,
 65\ldots56,
 90 \ldots 09,
 94\ldots49,
 98\ldots89 \nonumber
\}.
\end{align}
The elements above that are listed with an ellipsis are placeholders for the appropriately corresponding single variable family of curious numbers. For instance, we write $4443\ldots3444$ to denote the family
\[ 4443\ldots3444 \coloneqq \{444\underbrace{3\ldots3}_n444 : n \geq 0\}. \]

The assertion of Theorem \ref{main-thrm} is that $\mathcal{C} \cap \mathcal{S} = \set{0, \, 1, \, 4, \, 9, \, 121, \, 484, \, 676, \, 44944}$. Thus, by (\ref{CSpos}), it suffices to prove that there does not exist a perfect square in any of the following thirteen single variable families,
\begin{align*} \mathcal{F} \coloneqq \{
 &10 \ldots 01, \,
 14\ldots41, \,
 18\ldots81, \,
 40 \ldots 04, \,
 42\ldots24, \,
 442\ldots244, \,
 4443\ldots3444, \, \\
& 4447 \ldots 7444, \,
 52\ldots25, \,
 65\ldots56, \,
 90 \ldots 09, \,
 94\ldots49, \,
 98\ldots89 \}.
 \end{align*}

\section{Considering families via modular arithmetic} \label{modfamilies}

In this section, we use modular arithmetic considerations to prove that none of the families listed in $\mathcal{F}$ contain a square. Let's begin by considering $10\ldots01$. Observe that each number in this family is congruent to 2 modulo 3. As 2 is a quadratic non-residue modulo 3, it follows at once that no number in this family is a square. In fact, since each of the numbers in $40\ldots04$ and $90\ldots09$ is a square multiple of a number in $10\ldots01$, we deduce that these two families contain no squares as well.

Given the above, our problem is reduced to proving that the ten remaining families of $\mathcal{F}$ contain no squares. For each of these families, we record below the coefficients  $M_{a,b,m}$ and $N_{a,b,m}$, as defined in (\ref{MN}).
\renewcommand{\arraystretch}{1.20}
\begin{table}[H] 
\begin{tabular}{|lcccrr|} \hline
Family           & $a$ & $b$ & $m$ & $M_{a,b,m}$ & $N_{a,b,m}$ \\ \hline
$14\ldots41$     & 1   & 4   & 1   & $-31$       & 130         \\
$18\ldots81$     & 1   & 8   & 1   & $-71$       & 170         \\
$42\ldots24$     & 4   & 2   & 1   & 16          & 380         \\
$442\ldots244$   & 4   & 2   & 2   & 196         & 39800       \\
$4443\ldots3444$ & 4   & 3   & 3   & 996         & 3999000     \\
$4447\ldots7444$ & 4   & 7   & 3   & $-3004$     & 4003000     \\
$52\ldots25$     & 5   & 2   & 1   & 25          & 470         \\
$65\ldots56$     & 6   & 5   & 1   & 4           & 590         \\
$94\ldots49$     & 9   & 4   & 1   & 41          & 850         \\
$98\ldots89$     & 9   & 8   & 1   & 1           & 890 \\ \hline       
\end{tabular}
\caption{} \label{MN_table}
\end{table}
\vspace{-.4cm}

We rule out the possibility of perfect squares in each of the above families via two elementary lemmas. 

\begin{lemm} \label{mod-lem} If $M_{a,b,m}$ is a quadratic non-residue modulo $N_{a,b,m}$, then $a_mb_na_m$ is not a square for each $n \geq 0$.
\end{lemm}
\begin{proof} We prove the contrapositive. Suppose that $n_0$ is a nonnegative integer for which $a_mb_{n_0}a_m$ is a square. Then $a,b,m,n_0,$ and $y \coloneqq \sqrt{a_m b_{n_0}a_m}$ give a solution to the equation (\ref{alg-eq}). Reducing this equation modulo $N_{a,b,m}$, we find that 
\[ M_{a,b,m} \equiv (3y)^2 \pmod{N_{a,b,m}}. \]
Hence, $M_{a,b,m}$ is a quadratic residue modulo $N_{a,b,m}$.
\end{proof}

For instance, let's consider the family $14\ldots41$ in view of the above lemma. From Table \ref{MN_table}, we read that $(a,b,m) = (1,4,1)$,  $M_{a,b,m} = -31$, and $N_{a,b,m} = 130$. Note that $-31$ is a quadratic non-residue modulo $130$. Thus, we deduce that the family $14\ldots41$ contains no perfect squares. Applying Lemma \ref{mod-lem} to the data from Table \ref{MN_table} in this way for each of the ten families, we conclude that none of the following five families contain a perfect square:
\[ 14\ldots41, \, 18\ldots81, \, 4443\ldots3444, \, 4447\ldots7444, \, \text{and } 94\ldots49. \]

This leaves the five other families to consider. We do so via our next lemma, which requires some notation from elementary number theory. Let $M,N$ be  nonzero integers with $\gcd(M,N) = 1$. Then $M$ is invertible modulo $N$. Hence, we may speak of its multiplicative order modulo $N$, which is denoted by  $\o{ord}_N(M)$. Explicitly, $\o{ord}_N(M)$ is the least positive integer with the property that $M^{\o{ord}_N(M)} \equiv 1 \pmod{N}$.

\begin{lemm} \label{mod-lem2} Let $N$ be a positive integer with $\gcd(N,10) = 1$. If for each integer $k$ with $0 \leq k < \o{ord}_N(10)$, we have that $N_{a,b,m} \cdot 10^k + M_{a,b,m}$ is a quadratic non-residue modulo $N$, then $a_m b_n a_m$ is not a square for each $n \geq 0$.
\end{lemm}
\begin{proof} We prove the contrapositive. Suppose  that $n_0$ is such that $a_m b_{n_0} a_m$ is a square. Then $a,b,m,n_0,$ and $y \coloneqq \sqrt{a_m b_{n_0}a_m}$ give a solution to the equation (\ref{alg-eq}). Write $k_0$ to denote the integer for which $0 \leq k_0 < \o{ord}_N(10)$ and $k_0 \equiv n_0 \pmod{\o{ord}_N(10)}$. Then $10^{n_0} \equiv 10^{k_0} \pmod{N}$, so upon reducing (\ref{alg-eq}) modulo $N$, we find that
\[ N_{a,b,m} \cdot 10^{k_0} + M_{a,b,m} \equiv N_{a,b,m} \cdot 10^{n_0} + M_{a,b,m} \equiv (3y)^2 \pmod{N}. \]
Hence, $N_{a,b,m} \cdot 10^{k_0} + M_{a,b,m}$ is a quadratic residue modulo $N$.
\end{proof}

Using SageMath, we search for (and find) appropriate $N$ as in Lemma \ref{mod-lem2} for each of five remaining families. The relevant data is tabulated below. 
\renewcommand{\arraystretch}{1.20}
\begin{table}[H]
\begin{tabular}{|llll|}
\hline
Family         & $N$     & $\o{ord}_N(10)$ & $N_{a,b,m} \cdot 10^k + M_{a,b,m} \mod N$ \, for\, $0 \leq k < \o{ord}_N(10)$ \\ \hline
$42\ldots24$   & $999$   & $3$             & $396, \, 819, \, 54$                                                                \\ 
$442\ldots244$ & $77$    & $6$             & $33, \, 29, \, 66, \, 51, \, 55, \, 18$                                                      \\ 
$52\ldots25$   & $91$    & $6$             & $40, \, 84, \, 69, \, 10, \, 57, \, 72$                                                      \\ 
$65\ldots56$   & $13837$ & $8$             & $594, \, 5904, \, 3656, \, 8850, \, 5442, \, 12873, \, 4161, \, 63$                                \\ 
$98\ldots89$   & $1001$  & $6$             & $891, \, 893, \, 913, \, 112, \, 110, \, 90$                                                 \\ \hline 
\end{tabular}
\caption{}
\end{table}
\vspace{-.4cm}

Together with Lemma \ref{mod-lem2}, this data  proves that none of the five remaining families contain a perfect square. To highlight an example, let's consider the family $42\ldots24$. We read the values $396,\, 819,$ and $54$ from the fourth column of the first row of data. Because each of these is a quadratic non-residue modulo $N = 999$, Lemma \ref{mod-lem2} implies that the family in question contains no perfect squares.

\section{Considering families via elliptic curves} \label{elliptic} 
In this section, we give a method that determines the squares in a given single variable family  (as in \S\ref{modfamilies}). As we'll see, the squares are in one-to-one correspondence with the integral points of a specific form on certain elliptic curves. An \textit{elliptic curve} $E$ (defined over the rationals) is a projective curve given by an equation
\[ E : y^2 = x^3 + ax + b \]
for some $a,b \in \Q$ with nonzero discriminant $\Delta \coloneqq -16(4a^3 + 27b^2)$. The set of \textit{integral points} of $E$ is 
\[ E(\mathbb{Z}) \coloneqq \set{(x,y) \in \Z \times \Z : y^2 = x^3 + ax + b}\!. \]
This set is finite and fairly computable. See \cite{Si2009} for a thorough treatment of the theory of elliptic curves and specifically Chapter IX for a treatment of integral points. In what follows, we use the \texttt{IntegralPoints} command in Magma \cite{MAGM} to rigorously compute integral points on various elliptic curves.

We define some notation. Let $a,b,m$ be nonnegative integers with $1 \leq a \leq 9$, $0 \leq b \leq 9$, and $m\geq 1$. Then 
\[ \set{a_m b_n a_m : n \geq 0} \]
is a family of curious numbers of the sort that we considered in \S \ref{modfamilies}. We're interested in determining the set
\[
Q_{a,b,m} \coloneqq \set{n : a_mb_na_m \text{is a perfect square}}\!.
\]
For each $j \in \set{0,1,2}$, we set
\[ B_{a,b,m,j} \coloneqq  N_{a,b,m}^{2}\cdot10^{2j}\cdot M_{a,b,m} \]
and consider the elliptic curve
\[E_{a,b,m,j} : y^2 = x^3 + B_{a,b,m,j}. \]
Indeed, the above equation defines an elliptic curve since
\[    \Delta  =   -16(4\cdot0 +27\cdot B_{a,b,m,j}^{2}) = -16\cdot27\cdot B_{a,b,m,j}^{2} \neq 0.
\]
We're interested in the integral points of $E_{a,b,m,j}$ and, more  specifically, the subset
\[
L_{a,b,m,j} \coloneqq \set{(X,Y) \in  E_{a,b,m,j}(\Z) : (X,Y) =  ( N_{a,b,m}\cdot10^{j+k} , N_{a,b,m}\cdot10^{j}\cdot3y) \text{ for some integers } k,y \geq 0 }\!.
\]

\begin{prop}\label{ell} In the notation above, we have a bijection
\begin{align*} \Phi: Q_{a,b,m} &\longrightarrow L_{a,b,m,0} \cup L_{a,b,m,1} \cup L_{a,b,m,2}  \\
 n = 3k + j &\longmapsto  ( N_{a,b,m}\cdot10^{j+k} , N_{a,b,m}\cdot10^{j}\cdot 3\sqrt{a_m b_n a_m}),
\end{align*}
where we write $n = 3k + j$ with $j \in \set{0,1,2}$.
\end{prop}
\begin{proof} The map $\Phi$ is well-defined. Indeed, the decomposition $n = 3k +j$ with $j \in \set{0,1,2}$ is unique and we now verify that the image of $\Phi$ is contained in the stated codomain. If $n = 3k+j \in Q_{a,b,m}$, then
\[
(3 \sqrt{a_m b_n a_m})^2  = N_{a,b,m}\cdot10^{3k+j} + M_{a,b,m}
\]
holds by (\ref{alg-eq}). Multiplying through  by $N_{a,b,m}^{2}\cdot10^{2j}$ and regrouping, we find that
\begin{equation*}
(N_{a,b,m}\cdot10^{j}\cdot3\sqrt{a_mb_na_m})^{2} = (N_{a,b,m}\cdot10^{j+k})^{3} + N_{a,b,m}^{2}\cdot10^{2j}\cdot M_{a,b,m}.
\end{equation*}
Consequently, $\Phi(n) \in L_{a,b,m,j}$, establishing that the image of $\Phi$ is contained in the stated codomain. On reading this argument backwards, we deduce that $\Phi$ is surjective. What remains is to show that $\Phi$ is injective.

 For this, suppose that $n = 3k+j \in Q_{a,b,m}$ and $n'=3k'+j' \in Q_{a,b,m}$  are such that $\Phi(n) = \Phi(n')$. Then
\begin{equation} \label{ell_inj}
( N_{a,b,m}\cdot10^{j+k} , N_{a,b,m}\cdot10^{j}\cdot 3\sqrt{a_m b_n a_m}) = ( N_{a,b,m}\cdot10^{j'+k'} , N_{a,b,m}\cdot10^{j'}\cdot 3\sqrt{a_m b_{n'} a_m}).
\end{equation}
By comparing the first coordinates, we find that $j+k = j' + k'$. Note that since $a,m$ are nonzero, $a_mb_na_m$ is not divisible by $10$. Thus $\sqrt{a_mb_na_m}$ is not divisible by 10 (and nor is  $\sqrt{a_mb_{n'}a_m}$). Thus comparing the second coordinates of (\ref{ell_inj}), we find that $j = j'$. Since $j + k = j' + k'$ and $j = j'$, we have that $n = n'$.
\end{proof}

The above proposition, along with the  data from Table \ref{elliptic_table} (Appendix \ref{app}), gives an alternate proof that none of the thirteen families in
$\mathcal{F}$ contain a perfect square. 
To illustrate, let's consider the family
$42\ldots24$. Here, $(a,b,m) = (4,2,1)$ and the  corresponding elliptic curves for $j \in \{0,1,2\}$ are
\begin{align*}
     E_{4,2,1,0}: y^2 &= x^3 + 23104\cdot 10^{2}\\
     E_{4,2,1,1}: y^2 &= x^3 + 23104\cdot 10^{4}\\
     E_{4,2,1,2}: y^2 &= x^3 + 23104\cdot 10^{6}.
\end{align*}
Using Magma \cite{MAGM}, we compute their integral points:
\begin{align*}
     E_{4,2,1,0}(\Z) &= \{(80, \pm 1680), (0, \pm1520), (1520, \pm59280), (-76, \pm1368)\} \\
     E_{4,2,1,1}(\Z) &=  \{(0, \pm15200)\} \\
     E_{4,2,1,2}(\Z) &= \{(0, \pm152000)\}.
\end{align*}
None of these points are of the form $( N_{a,b,m}\cdot10^{j+k} , N_{a,b,m}\cdot10^{j}\cdot3y)$ for  nonnegative integers $y,k$. Hence, 
\[ L_{4,2,1,0} = L_{4,2,1,1} = L_{4,2,1,2} = \emptyset. \]
So, by Proposition \ref{ell}, we conclude that the family $42\ldots24$ contains no perfect squares.

\bibliographystyle{amsplain}
\bibliography{Curious_Squares}

\newpage

\appendix

\section{} \label{app}

\renewcommand{\arraystretch}{1.08}
\begin{table}[H] 
\begin{tabular}{|lllp{95mm}|} \hline
Family            & $j$ & $B_{a,b,m,j}$                                                                  & Integral points of $E_{a,b,m,j}: y^2 = x^3 + B_{a,b,m,j}$ up to sign     \\ \hline
$10\ldots01$      & 0   & $ 729 \cdot 10^2$                                                            & \{($-36$, 162), (0, 270), (40, 370), (45, 405), (180, 2430), (216, 3186), (23940, 3704130)\}                                                                                    \\
                   & 1   & $ 729\cdot 10^4$                                                          & \{(0, 2700)\}                                                                                                                                                                  \\
                  & 2   & $ 729\cdot 10^6$                                                        & \{($-900$, 0), (1800, 81000), (0, 27000)\}                                                                                                                                       \\ \hline
$14\ldots41$     & 0   & $ -5239\cdot 10^2$                & \{(100, 690), (140, 1490), (160, 1890), (1589, 63337), (28261, 4750959)\}                                                                                                      \\
                    & 1   & $ -5239\cdot 10^4$                                                         & \{(376, 876)\}                                                                                                                                                                 \\
                    & 2   & $ -5239\cdot 10^6$                                                       & \{(3500, 194000)\}                                                                                                                                                             \\ \hline
$18\ldots81$       & 0   & $ -20519\cdot 10^2$                                                          & \{(960, 29710)\}                                                                                                                                                               \\
                  & 1   & $ -20519\cdot 10^4$                                                        &$\emptyset$                                                                                                                                                                           \\
                   & 2   & $ -20519\cdot 10^6$                                                      &$\emptyset$                                                                                                                                                                           \\ \hline
$40\ldots04$      & 0   & $ 46656\cdot 10^2$                                                          & \{($-144$, 1296), ($-135$, 1485), (0, 2160), (160, 2960), (180, 3240),  (720, 19440), (864, 25488), (95760, 29633040)\}                                                             \\
                   & 1   & $ 46656\cdot 10^4$                                                        & \{(0, 21600)\}                                                                                                                                                                 \\
                   & 2   & $ 46656\cdot 10^6$                                                      & \{ ($-3600$, 0), (0, 216000), (7200, 648000)\}                                                                                                                                    \\ \hline
$42\ldots24$       & 0   & $ 23104\cdot 10^2$                                                          & \{($-76$, 1368), (0, 1520), (80, 1680), (1520, 59280)\}                                                                                                                          \\
                     & 1   & $ 23104\cdot 10^4$                                                        & \{(0, 15200)\}                                                                                                                                                                 \\
                    & 2   & $ 23104\cdot 10^6$                                                      & \{(0, 152000)\}                                                                                                                                                                \\ \hline
$442\ldots244$   & 0   & $ 31047184\cdot 10^4$                                                     & \{($-4975$, 432825), (0, 557200), (5600, 697200), (44576, 9427824)\}                                                                                                             \\
                    & 1   & $ 31047184\cdot 10^6$                                                   & \{(0, 5572000), (8959776, 26819194976)\}                                                                                                                                       \\
                      & 2   & $ 31047184 \cdot 10^8$                                                 & \{($-84000$, 50120000), (0, 55720000), (1671600, 2161936000)\}                                                                                                                   \\ \hline
$4443\ldots3444$  & 0   & $ 15928032996\cdot 10^6$                                                & \{(198400, 154070000)\}                                                                                                                                                        \\
                     & 1   & $ 15928032996\cdot 10^8$                                              & \{($-356000$, 1244060000)\}                                                                                                                                                 \\
                     & 2   & $ 15928032996\cdot 10^{10}$ &$\emptyset$                                                                                                                                                                      \\ \hline
$4447\ldots7444$    & 0   & $-48136123036\cdot 10^6$                                                 &$\emptyset$                                                                                                                                                                   \\
                      & 1   & $-48136123036\cdot 10^8$   &$\emptyset$                                                                                                                                                                   \\
                     & 2   & $-48136123036\cdot 10^{10}$ &$\emptyset$                                                                                                                                                                   \\ \hline
$52\ldots25$    & 0   & $  55225\cdot 10^2$                                                         & \{(0, 2350)\}                                                                                                                                                                  \\
                    & 1   & $ 55225\cdot 10^4$                                                        & \{(0, 23500)\}                                                                                                                                                                 \\
                   & 2   & $ 55225\cdot 10^6$                                                      & \{(0, 235000)\}                                                                                                                                                                \\ \hline
$65\ldots56$    & 0   & $  13924\cdot 10^2 $           & \{(0, 1180), (80, 1380), (944, 29028)\}                                                                                                                                        \\
                     & 1   & $ 13924\cdot 10^4 $            & \{(0, 11800)\}                                                                                                                                                                  \\
                      & 2   & $ 13924\cdot 10^6 $        & \{($-2400$, 10000), (0, 118000), (4425, 317125), (751296, 651203344)\}                                                                                                                 \\ \hline
$90\ldots09$    & 0   & $ 531441\cdot 10^2$                                                         & \{ ($-324$, 4374),  (0, 7290), (360, 9990), (405, 10935), (1620, 65610), (1944, 86022),  (215460, 100011510)\}                                                                      \\
                      & 1   & $ 531441\cdot 10^4 $        & \{(0, 72900)\}                                                                                                                                                                 \\
                    & 2   & $ 531441\cdot 10^6$         & \{($-8100$, 0), (0, 729000), (16200, 2187000)\}                                                                                                                                  \\ \hline
$94\ldots49$    & 0   & $ 296225\cdot 10^2$             & \{($-200$, 4650), ($-100$, 5350), (349, 8493), (10300, 1045350)\}                                                                                                                  \\
                    & 1   & $ 296225\cdot 10^4$                                                       & \{($-800$, 49500),  (200, 54500), (625, 56625), (11416, 1220964)\}                                                                                                                \\
                    & 2   & $ 296225\cdot 10^6 $        &$\emptyset$                                                                                                                                                                           \\ \hline
$98\ldots89$    & 0   & $  7921\cdot 10^2 $            & \{(0, 890)\}                                                                                                                                                                   \\
                   & 1   & $ 7921\cdot 10^4$                                                         & \{($-400$, 3900), (0, 8900), (1424, 54468)\}                                                                                                                                 \\
                  & 2   & $ 7921\cdot 10^6$            & \{(0, 89000), (8400, $-775000$)\}  \\ \hline                                                                                                                                       
\end{tabular}
\caption{. Integral points data for \S \ref{elliptic} } \label{elliptic_table}
\end{table}

\end{document}